\documentclass[letterpaper, 10 pt, conference]{ieeeconf}  

\IEEEoverridecommandlockouts                              
\usepackage{amssymb,amsmath,color}
\usepackage{url}
\usepackage{amsfonts}
\usepackage{csquotes}
\usepackage[mathcal]{eucal}
\usepackage{empheq}
\usepackage{comment}
\usepackage{hyperref}

\DeclareMathAlphabet\mathbfcal{OMS}{cmsy}{b}{n}

\newcommand{\BEAS}{\begin{eqnarray*}}
\newcommand{\EEAS}{\end{eqnarray*}}
\newcommand{\BEA}{\begin{eqnarray}}
\newcommand{\EEA}{\end{eqnarray}}
\newcommand{\BEQ}{\begin{equation}}
\newcommand{\EEQ}{\end{equation}}
\newcommand{\BAS}{\begin{align*}}
\newcommand{\EAS}{\end{align*}}
\newcommand{\BIT}{\begin{itemize}}
\newcommand{\EIT}{\end{itemize}}
\newcommand{\BNUM}{\begin{enumerate}}
\newcommand{\ENUM}{\end{enumerate}}
\newcommand{\BA}{\begin{array}}
\newcommand{\EA}{\end{array}}

\newcommand{\argmin}{\mathop{\rm argmin}}

\newcommand{\tr}{\mathop{ \rm tr}}

\newcommand{\rb}{\mathbb{R}}

\newcommand{\dv}[2]{{\frac{\partial #1}{\partial #2}}}
  
\newcommand{\BlackBox}{\rule{1.5ex}{1.5ex}}  

\newtheorem{lemma}{Lemma}

\newtheorem{proposition}{Proposition}
\newtheorem{corollary}{Corollary}
\newtheorem{remark}{Remark}

\usepackage{tikz} 
\tikzstyle{state}=[shape=circle,draw=black!70,minimum size = 1cm,node distance = 3cm]
\tikzstyle{observation}=[shape=rectangle,draw=black!50]
\tikzstyle{lightedge}=[<-,dotted]
\tikzstyle{mainstate}=[state,thick]
\tikzstyle{mainedge}=[<-,thick]

\def \E{{\mathbb E}}

\newcommand{\bivec}[2]{\Big( \! \! \begin{array}{c} {#1} \\ {#2} \end{array} \! \! \Big)}
 \newcommand{\bimat}[4]{\bigg(  \! \! \begin{array}{c@{\hspace{2mm}}c} {#1} &  {#2} \\ {#3} &  {#4} \end{array}   \! \! \bigg)}

\title{\LARGE \bf
Variational Dynamic Programming for Stochastic Optimal Control
}

\author{ \parbox{2 in}{\centering Marc Lambert\\
         DGA/CATOD\\
         INRIA - ENS - PSL \\
         {\tt\small marc.lambert@inria.fr}} \hspace*{ 0.4 in}
         \parbox{2 in}{ \centering Francis Bach\\
         INRIA - ENS - PSL  \\
         {\tt\small francis.bach@inria.fr}
         \hspace*{ 0.2 in}
         \parbox{2.2 in}{ \centering Silvère Bonnabel\\
         MINES Paris PSL\\
         {\tt\small silvere.bonnabel@mines-paristech.fr}}
        }
}

\author{Marc Lambert, Francis Bach and Silvere Bonnabel
    \thanks{M. Lambert and F. Bach are with Inria, Departement d’Informatique de l’Ecole Normale Superieure, PSL Research University. 
    S. Bonnabel is with Mines Paris PSL, PSL Research University, Centre for Robotics.  
    }%
}

\begin{document}

\maketitle
\thispagestyle{empty}
\pagestyle{empty}

\begin{abstract}
We consider the problem of stochastic optimal control, where the state-feedback control policies take the form of a probability distribution and where a penalty on the entropy is added.  By viewing the cost function as a Kullback-Leibler (KL) divergence between two joint distributions, we bring the tools from variational inference to bear on our optimal control problem. This allows for deriving a dynamic programming principle, where the value function is defined as a KL divergence again. We then resort to Gaussian distributions to approximate the control policies and apply the theory to control affine nonlinear systems with quadratic costs. This results in closed-form recursive updates, which generalize LQR control and the backward Riccati equation. We illustrate this novel method on the simple problem of stabilizing an inverted pendulum.  
\end{abstract}

\section{INTRODUCTION} \label{section1}

In this article, we consider a stochastic dynamical system governed in discrete time by a known Markovian transition $p(x_{k+1}| x_k, u_k)$ where  $x_k  \in \mathbb{R}^d$  is the current state and $u_k \in \mathbb{R}^m$ is the control variable with $m \leq d$. The initial state $x_0$ is supposed to be known. To precisely state our positioning and our contributions, we need to introduce the notation of the classical setup. For the expectations $\int p(x) f(x) dx$, we use the notation $\E[f(x)]$ or $\E_{p(x)}[f(x)]$.

\subsection{Basics of Stochastic Optimal Control}
To control   future states starting from $x_0$ and over a finite horizon~$K$,  we first  consider the stochastic finite horizon optimal control problem in discrete time: 
\BEA
 \hspace{-0.05cm} \underset{u_0,\dots,u_{K-1}}{\min} \E \Big[ \sum_{k=0}^{K-1} \ell_k(x_k,u_k)  + L_K(x_K)\Big], \label{SOC}
\EEA
where the expectation is under the stochastic trajectories starting from $x_0$; $\ell_k$ denote the stage   cost functions for $0\leq k\leq K-1$ and $L_K$ the final cost function. These functions are supposed to be continuous. 

In this context, the goal is typically to derive causal state-feedback control policies $u_k=\varphi_k(x_0,\dots,x_{k-1})$ so as to solve \eqref{SOC}. A key result in that regard is that of dynamic programming, which states that one may define a  value function $V_K$ based on the final cost $V_K(x_K):=L_K(x_K)$, and then define $V_k$ through the backward  recursion:
\BEA
V_k(x_k)=\underset{v}{\min} \ \ell_k(x_k,v)\! +\! \E_{ p(x_{k+1}\mid x_k,v)}   \Big[ V_{k+1}(x_{k+1}) \Big].
\EEA
$V_k$ is a function defined over the entire state space, termed ``cost-to-go," also called value function, that encapsulates the minimum cost when  (deterministically) starting from $x_k$. This yields an optimal causal state-feedback policy
\begin{equation*}
    \varphi^*(x_k)=\underset{v}{\argmin} \quad \ell_k(x_k,v) + \E_{ p(x_{k+1}\mid x_k,v)}   \Big[ V_{k+1}(x_{k+1})\Big],
\end{equation*}
defining a sequence of control inputs that minimize \eqref{SOC}. 
 
\subsection{Probabilistic State-Feedback Policy}
A somewhat different problem arises when the control policy is taken as a   \emph{probability distribution} (a density) of the form $p(u_k|x_k)$ instead of $u_k=\varphi_k(x_k)$.     Letting   $z_{0:K} := ( u_0,x_1,u_1,\dots, x_{K-1},u_{K-1},x_K)$, its density then decomposes using the Markov property as follows:
\BEA
p(z_{0:K}|x_0) =  \prod_{k=0}^{K-1} p(u_k|x_k) p(x_{k+1}|x_k,u_k). \label{graph:eq}
\EEA

As is common in graphical models, we will overload notation by letting letter  $p$ denote all probability densities. The associated graphical model is shown in figure \ref{PGMforward}.

\begin{figure}[thpb]
\centering
\begin{center}
\includegraphics[scale=0.75]{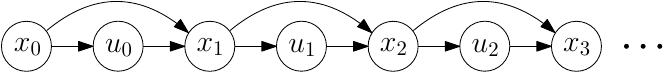}
\end{center}
\vspace*{-.4cm}
\caption{Graphical model associated with \eqref{graph:eq}, where $u_k|x_k$ is a probability distribution.}
\label{PGMforward}
\end{figure}
 
This turns \eqref{SOC} into the alternative control problem 
\BEA
\underset{p(u_k|x_k),0\leq k\leq  K-1}{\min} \!\E_{p(z_{0:K}|x_0)} \!\Big[ \sum_{k=0}^{K-1} \!\ell_k(x_k,u_k)  \!+\! L_K(x_K)\Big].\!\! \label{SOC2}
\EEA 

As is, the argmin over policies $p(u_k|x_k)$ consists of Dirac distributions $\varphi^*(x_k)$, and one recovers the optimal deterministic state-feedback policy for \eqref{SOC}. To obtain a random policy, one can add to the cost \eqref{SOC2} a penalty of magnitude $\varepsilon$ on the negentropy of the policy function $\int \! p(u_k|x_k) \log p(u_k|x_k) du_k$,   as proposed in \cite{Ziebart2008,Haarnoja2018}, leading to the following regularized problem :
  \begin{align}
\underset{p(u_k|x_k),0\leq k\leq  K-1}{\min} & \E_{p(z_{0:K}|x_0)} \Big[ \sum_{k=0}^{K-1} \Big(\ell_k(x_k,u_k) \label{MaxEnt}  \\
&+ \varepsilon \log p(u_k|x_k) \Big)  + L_K(x_K) \Big], \nonumber 
  \end{align}
 where $\varepsilon$ is a ``temperature" parameter.  Note that, as  $\varepsilon \rightarrow 0$, we recover the deterministic policy \eqref{SOC}.    

It turns out the dynamic programming principle carries over to the problem above, see \cite{Ziebart2010,Oswin2022}. Starting from   $V_K^{(c)}(x_K):=L_K(x_K)$, we may define a cost-to-go through the backward recursion: 
\BEA
V_k^{(c)}(x_k)& :=& \underset{p(u_k|x_k)}{\min} \E_{p(u_k|x_k) p(x_{k+1}|x_k,u_k)} \Big[ \ell_k(x_k,u_k) \nonumber\\
&& \hspace*{.5cm} + \varepsilon \log p(u_k|x_k) + V_{k+1}^{(c)}(x_{k+1})\Big],   \label{Bellman-Vc}
\EEA
where superscript  $(c)$ stands for conditional entropy regularization. 
An appealing aspect of this conditional entropy regularization is that the policy 
can be expressed exactly \cite{Ziebart2010}:
\BEA
p^*(u_k|x_k) \propto \exp \Big[\! - \! \frac{1}{\varepsilon} \big(
\ell_k(x_k,u_k) \!+\! \E[  V_{k+1}^{(c)}(x_{k+1})]\big) \Big]. \! \label{OptimPolicyMaxEnt}
\EEA

\subsection{Considered Problem}
Instead of penalizing the negentropy of the state-feedback policy, as done in \eqref{MaxEnt}, we propose to consider the classical stochastic optimal control problem \eqref{SOC} with an additional penalty $+\varepsilon \log p(z_{0:K}|x_0)$, that is, penalizing the negentropy  of the \emph{full} joint distribution $ p(z_{0:K}|x_0)$. Given the decomposition  \eqref{graph:eq}, this means we consider the problem:
\BEA
&\underset{p(u_k|x_k),0\leq k\leq  K-1}{\min} \E_{p(z_{0:K}|x_0)} \Big[ \sum_{k=0}^{K-1} \Big(\ell_k(x_k,u_k) \label{KLcontrolprev0}  \\
&+ \varepsilon \log p(u_k|x_k) + \varepsilon \log p(x_{k+1}|x_k,u_k) \Big)  + L_K(x_K) \Big]. \nonumber 
\EEA
Starting from   $V_K^{(f)}(x_K):=L_K(x_K)$, we now define a cost-to-go through the backward recursion 
\BEA
   &V_k^{(f)}(x_k)=\underset{p(u_k|x_k)}{\min} \E_{p(u_k|x_k) p(x_{k+1}|x_k,u_k)} \Big[ \ell_k(x_k,u_k) \label{Dyna:eq}\\
   &+ \varepsilon \log p(u_k|x_k)+\varepsilon \log p(x_{k+1}|x_k,u_k)+ V_{k+1}^{(f)}(x_{k+1})\Big], \nonumber
\EEA
where superscript  $(f)$ stands for full entropy regularization.   \eqref{Dyna:eq} consists of \eqref{Bellman-Vc} plus the term $ \varepsilon \log p(x_{k+1}|x_k,u_k)$.  

This variational dynamic programming principle  minimizes indeed the total loss \eqref{KLcontrolprev0} which is exactly $V_0^{(f)}(x_0)$.
Moreover, the problem \eqref{KLcontrolprev0} can be recast as a Kullback-Leibler (KL) divergence between two joint distributions. The \textit{exact} optimal solution to this problem can be analytically expressed using variational inference methods \cite{Wainwright08}, thereby generalizing equation \eqref{OptimPolicyMaxEnt}.

We briefly discuss the rationale behind this joint entropic regularization. First, penalizing the negentropy of $p(u_k|x_k)$ fosters distributions of greater entropy, that is, ``flatter'' distributions, which result in softer controls. 
Our formulation goes further by explicitly penalizing the entropy of the states through the term $ \varepsilon \log p(x_{k+1}|x_k,u_k)$, leading to further exploration of the state space. The formulation is compatible with stochastic control or model-based reinforcement learning where we have access to   $p(x_{k+1}|x_k,u_k)$.

\subsection{Variational Approximation of the Optimal Control}
In practice, determining the exact optimal policy for either conditional negentropy penalization \eqref{MaxEnt}  or full negentropy penalization \eqref{KLcontrolprev0} is generally difficult, as unnormalized formulas such as   \eqref{OptimPolicyMaxEnt} make the computation of simple features such as the mean and covariance  typically intractable. Hence, it is desirable to approximate the optimal policy \eqref{OptimPolicyMaxEnt}  by a parametric family of distributions $q(u_k|x_k)$, like a Gaussian distribution, considering the variational problem \cite{Haarnoja2018}:
\BEA
&\underset{q(u_k|x_k)}{ \arg \min} & {\rm KL}\big(q(u_k|x_k)\| p^*(u_k|x_k)\big) \label{SACpolicy} 
\EEA
where $\rm KL$ is the unnormalized Kullback-Leibler divergence defined by $ {\rm KL}(q(y)\|p(y)):=\int q(y)\log q(y)dy-\int q(y)\log p(y)dy.$ Since the optimal policy depends on the cost-to-go \eqref{Bellman-Vc}, this function needs also to be estimated  leading to the soft-actor critic algorithm \cite{Haarnoja2018} where the policy and cost-to-go are updated alternately.
 However, we propose to consider instead a joint approximation of both the policy \emph{and} the cost-to-go. In particular, the left KL divergence between a parametric family $q(x_k)$ and the Gibbs distribution associated to the cost-to-go \eqref{Dyna:eq}:  
\BEA
&\underset{q(x_k)}{\min}  &\quad {\rm KL}\big(q(x_k)\| \exp(-V_k^{(f)}(x_k)/\varepsilon)\big), \label{fit:crit}
\EEA
can be rewritten as a variational problem on the joint model $q(u_k,x_k)=q(u_k|x_k)q(x_k)$. In this paper, we consider this joint Gaussian distribution and derive a recursive update for its parameters. In particular we show the precision matrix of the Gaussian marginal  $q(x_k)$ follows an implicit backward equation that generalizes the   Riccati equation from linear quadatric control (LQR).

\subsection{Related Works}
\paragraph{On Maximum Entropy Policy}
In the context of reinforcement learning, the maximum entropy policy is used to enhance exploration, as with the soft actor-critic algorithm \cite{Haarnoja2018} discussed above. In stochastic control, the maximum entropy policy problem  \eqref{MaxEnt} has been introduced to learn the cost function from an observed policy. Several applications are discussed in  \cite{Ziebart2010} like the two-player game where random policies provide unpredictable trajectories. Gaussian approximation of the optimal maximum entropy policy \eqref{OptimPolicyMaxEnt} is considered in the context of differential dynamic programming \cite{watson20a,Oswin2022}. To achieve this the dynamics are linearized, and the cost function is approximated locally with a quadratic cost. Our variational approach avoids these linearizations. 
\paragraph{On the KL formulation of Stochastic Optimal Control}
Our variational formulation  \eqref{KLcontrolprev0},\eqref{KLcontrol} differs from previous KL formulations proposed in stochastic control. In ``KL control" \cite{Todorov2009}, the KL penalizes a discrepancy between the controlled dynamic $p(x_{k+1}|x_k,u_k)$ and the passive one $p(x_{k+1}|x_k,u_k=0)$, and thus serves as an indirect penalty on input usage. 
``KL control" is also related to inference in graphical models in \cite{Kappen2012}, path integral and risk sensitivity \cite{Theodorou2012, Fleming06} where a \textquote{log-sum-exp} variant of the Bellman recursion was proposed. A KL cost setting was also proposed as an extension of the Schr\"odinger bridge problem for stochastic control, see  \cite{Chen2021}.  All these approaches do not use the regularization with the entropy of the policy and do not provide a random policy. 
A KL formulation close to \eqref{KLcontrolprev0},\eqref{KLcontrol} with a random policy was proposed in \cite{Rawlik07} and related to a cost-to-go regularized with the entropy of the policy. In this setting, the rewards are seen as observations, and the optimal policy is computed by variational inference. The case of control-affine inputs is discussed, but no closed-form updates have been derived. In \cite{Giovanni22}, a KL divergence between two joint distributions—representing the learned dynamic and a reference dynamic—was proposed to design control policies from demonstrations, providing an explicit solution for the optimal policy. However, the reference dynamic was not explicitly linked to any cost function.

\subsection{Main Contributions and Paper Organization}
This paper aims to address problem \eqref{KLcontrolprev0}, or ~equivalently \eqref{Dyna:eq}, using the variational inference (VI) framework to derive both theoretical and practical approximate solutions.  Our contributions are listed below:
\begin{itemize}
\item In Section \ref{section2}, we turn \eqref{Dyna:eq} into variational dynamic programming, by re-writing it as a Kullback-Leibler (KL) minimization problem. We give the exact formulation of the optimal policy and cost-to-go.
\item Still in Section \ref{section2}, we show how one may approximate \emph{jointly} the policy and the value function with a    distribution $q(x_k,u_k)$. Restricting to the class of  Gaussian distributions, we may immediately deduce the approximated policy $q(u_k|x_k)$ and approximated value function $q(x_k)$. 
\item In Section \ref{section3}, we consider the particular case of control-affine nonlinear dynamics with a quadratic cost function, and we show we can derive explicit formulas for the optimal parameters of the Gaussian $q(x_k,u_k)$, leading to novel variational backward Riccati equations.
\item In Section \ref{section4}, we compare the obtained policy and linearized LQR for the stabilization of a noisy (inverted) pendulum around an equilibrium point and show how our policy increases entropy while stabilizing. 
\end{itemize} 

\section{Variational Dynamic Programming } \label{section2}
\subsection{A Variational Dynamic Programming Principle}
Following \cite{Todorov08}, \cite{Toussaint2009}, we can rewrite the cost as follows: $\ell_k(x_k,u_k)  = -  {\varepsilon} \log r(x_k,u_k)$, and $L_K(x_K)  = - {\varepsilon}  \log r(x_K)$ where $\varepsilon > 0$ is the same temperature parameter introduced in equation \eqref{MaxEnt}. Here, $r(x_k,u_k)$ and $r(x_K)$ can be interpreted as reward distributions, taking the form of an unnormalized Gibbs distributions. One may then associate an unnormalized joint distribution with the cost function. Given that the initial state $x_0$ is known, this joint distribution can be factored as follows:
\BEA
r(z_{0:K}|x_0) = \Big( \prod_{k=0}^{K-1} r(x_k,u_k) \Big) r(x_K).
\EEA 
Problem  \eqref{KLcontrolprev0}, which we address, then rewrites:
\BEA
&\underset{p(u_k|x_k),0 \leq k \leq K-1}{\min} \varepsilon {\rm KL}(p(z_{0:K}|x_0)\|r(z_{0:K}|x_0)). \label{KLcontrol}
\EEA
The dynamic programming recursion \ref{Dyna:eq} can also translated into a KL minimization problem. Using the Gibbs formulation for the loss $\ell_k(x_k,u_k)=-\varepsilon \log r(x_k,u_k)$, Equation \eqref{Dyna:eq} rewrites:
\begin{alignat}{2}
&V_k^{(f)}(x_k) \label{Dyna:eq2}\\
&\hspace{-0.1cm} = \underset{p(u_k|x_k)}{\min} \varepsilon {\rm KL}(p(x_{k+1}|u_k,x_k)p(u_k|x_k)\|r(x_k,u_k)\phi(x_{k+1}))\nonumber, 
\end{alignat}
where we let: \begin{align}
    \phi(x_{k+1}):=\exp{(-V_{k+1}^{(f)}(x_{k+1})/\varepsilon)}. \label{phi:def}
\end{align}
This problem can be solved exactly using the properties of KL divergences. 

\subsection{The ``Exact" Optimal Policy}
We now give our first main result, which generalizes \eqref{OptimPolicyMaxEnt}; the proof is postponed to Appendix \ref{AppendixA}.

\begin{proposition} \label{Proposition1}
The solution to problem \eqref{Dyna:eq2} is given by:
\begin{alignat*}{2}
&p^*(u_{k}|x_k) = \frac{1}{\phi(x_k)} \exp \big(- Q_k^{f} (u_k,x_k)\big) \\
&Q_k^{f} (u_k,x_k)={\rm KL}(p(x_{k+1}|u_k,x_{k})\|r(u_k,x_k)\phi(x_{k+1})).
\end{alignat*}
The optimal cost-to-go $V_k^{(f)}(x_k)$ depends on $\phi(x_k)$, the partition function of $p^*(u_{k}|x_k)$ as follows: 
\begin{alignat*}{2}
V_k^{(f)}(x_k)&=-\varepsilon \log \phi(x_k) \\
&= -\varepsilon \log \int \exp \big(- Q_k^{f} (u_k,x_k) \big) du_k,
\end{alignat*}
such that $V_k^{(f)}(x_k)$ takes a \textquote{log-sum-exp} form. \BlackBox
\end{proposition}

We now have a fully defined Bellman-like recursion to solve the entropy-regularized problem \eqref{KLcontrol}. Albeit interesting at a theoretical level, the latter formula is difficult to apply in practice. Indeed, akin to Bayesian inference, it is generally not analytically tractable, so one has to resort to approximations.  

\subsection{Joint Variational Approximation}

Given a well-chosen family $\mathcal{P}_k^{(u)}$ of densities $p(u_k|x_k)$, our goal is to maximize the same objective function but with this added constraint. This leads to the same recursion as~\eqref{Dyna:eq}, but with the extra constraint
that $p(u_k|x_k) \in \mathcal{P}_k^{(u)}$. 

The problem is now the representation of the resulting function $V^{(f)}(x_k)$, which needs to be simple enough to be propagated. Sticking with our representation of the costs in the form $V_k^{(f)}(x_k) = - \varepsilon \log \phi(x_k)$, we may approximate  in turn $\phi(x_k)$, using a well-chosen family $\mathcal{P}_k^{(x)}$ of (unnormalized) probability distributions. This is achieved via the variational approximation problem  introduced in  \eqref{fit:crit}.

Interestingly, it turns out that, by combining the problem of searching (restricted) optimal control policies $q(u_k|x_k)$ over $\mathcal{P}_k^{(u)}$ using \eqref{Dyna:eq2} and optimal approximations of the cost-to-go $q(x_k)$ over $\mathcal{P}_k^{(x)}$ using \eqref{fit:crit}, we recover a similar KL minimization problem, but over the \emph{joint} distribution $ q(u_k|x_k) q(x_k)$. Let us indeed substitute \eqref{Dyna:eq2} into problem~\eqref{fit:crit}:
\begin{align}
    &\min_{q(x_k) \in \mathcal{P}_k^{(x)}} \varepsilon {\rm KL}( q(x_k) \| \exp( - V_k^{(f)}(x_k) / \varepsilon ) ) \nonumber
    \\&=\!\!\min_{q(x_k) \in \mathcal{P}_k^{(x)}} \varepsilon \int q(x_k)\log(q(x_k))dx_k\!+\!\!\int\! q(x_k)V_k^{(f)}(x_k)dx_k \nonumber 
     \\&=\min_{q(x_k) \in \mathcal{P}_k^{(x)}} \min_{q(u_k|x_k) \in \mathcal{P}_k^{(u)}} \label{VI-KL}   \\&\quad \varepsilon  {\rm KL}(\underbrace{q(x_k)q(u_k|x_k)}_{\text{joint}}p(x_{k+1}|x_{k},u_k)\|r(u_k,x_{k})\phi(x_{k+1})),\nonumber
\end{align}

where we have formed the joint entropy using the relation 
$H(q(x_k))+\int q(x_k) H(q(u_k|x_k)p(x_{k+1}|x_{k},u_k)) dx_k=H(q(x_k)q(u_k|x_k)p(x_{k+1}|x_{k},u_k))$ where $H(p(x)):=-\int p(x) \log p(x) dx$ is the entropy. This relation comes from the fact that $q(u_k|x_k)$ and $p(x_{k+1}|x_{k},u_k)$ are normalized. 
Of course, in practice, $\phi(x_{k+1})$
is approximated in the previous step by $q(x_{k+1})$, and should be replaced accordingly. 
 
This is quite practical when $ q(u_k|x_k) q(x_k)$ ends up being in a simple family, so that we are faced with the usual variational approximation consisting of a (left) KL minimization of a function of $(x_k,u_k)$. 

\subsection{Gaussian   Approximation}\label{Gauss:sec}

Although various approximating families can be leveraged, the simplest is arguably to use a joint Gaussian distribution    $q(x_k,u_k)=q(u_k | x_k) q(x_k)$. We then intend to learn its parameters based on the obtained Bellman-like equations and immediately benefit from Gaussian conditioning formulas to recover $q(x_k)$ and $q(u_k | x_k)$. We parametrize this joint Gaussian as follows
 \begin{alignat}{2}
q(x_k,u_k) &:= \mathcal{N}\big(\mu_k, \Sigma_k \big) \label{Qvalue}\\
&=\mathcal{N}\bigg(\bivec{\alpha_k}{\beta_k}, \varepsilon \bimat{P_k^{-1}}{P_k^{-1} K_k^\top}{K_k P_k^{-1}}{K_k P_k^{-1} K_k^\top + S_k^{-1}} \bigg) \nonumber\\
q(x_k) &= \mathcal{N}(\alpha_{k},  {\varepsilon} P_{k}^{-1})\\
q(u_k | x_k) &=\mathcal{N}( \beta_{k} + K_k(x_k - \alpha_k) ,  {\varepsilon} S_{k}^{-1}). 
\end{alignat}

\begin{remark}
    The rationale for using a joint Gaussian distribution is as follows. With this choice, the feedback policy  $q(u_k | x_k)$ takes the form of a distribution dispersed around its mean. More interestingly, the fact that $q(x_k)$ is Gaussian means we use a quadratic approximation for the value function, as $q(x_k)$ is an approximation of $\phi(x_k)$. Finally, note that this choice of joint distribution enforces a linear feedback $\beta_{k} + K_k(x_k - \alpha_k)$ in terms of the mean of $q(u_k | x_k)$. 
\end{remark}

\section{Variational Backward Riccati Equation} \label{section3}
We have seen that the original entropy-regularized stochastic optimal control problem \eqref{KLcontrolprev0} is amenable to the dynamic programming recursion \eqref{VI-KL}, when constraining the policy and value distribution to lie in some approximating families. If we opt for the  Gaussian family \eqref{Qvalue}, problem \eqref{VI-KL} becomes an optimization problem over the parameters   $\alpha_{k}$, $\beta_{k}$, $S_k$,  $P_k$ and $K_k$. Following our previous work on recursive variational Gaussian approximation \cite{rvga}, we seek to derive (backward) recursive equations for those parameters. To achieve this, we focus on control-affine systems where the control inputs enter linearly into the dynamics. This includes various mechanical systems,   such as the cart-pole system or the two-link robot of \cite{spong2005underactuated}. Opting for quadratic cost functions, we then obtain equations that generalize the backward Riccati equation from LQR control.  

\subsection{Nonlinear Control-Affine Dynamics}
We focus on dynamics of  the following form:
\begin{equation}\label{dyna3:eq}
  x_{k+1}=f(x_k)+Bu_k+\nu_k, \qquad \nu_k\sim \mathcal N(0,C),
\end{equation}
where $B \in  \mathcal{M}_{d \times m}(\mathbb{R})$ and $C \in  \mathcal{M}_{d \times d}(\mathbb{R})$; $C \succ 0$.
It entails that $p(x_{k+1}|x_k,u_k) = \mathcal{N}(x_{k+1}|f(x_k)+Bu_k,  C)$. We also  choose to work with quadratic costs with $Q,P_K \in  \mathcal{M}_{d \times d}(\mathbb{R})$; $Q,P_K \succ 0$:
\begin{align}
    \ell(x_k,u_k)&= \textstyle \frac{1}{2}(x_k-x_k^*)^TQ(x_k-x_k^*)+\frac{1}{2}u_k^TRu_k, \nonumber \\  L(x_K)&=\textstyle \frac{1}{2}(x_K-x_K^*)^TP_K(x_K-x_K^*),\label{cost:eq:q}
\end{align}
where $x_k^*$ for $k=1,\dots, K$ is the reference trajectory. Our results will remain valid if the matrixes $B, C, Q, R$ depend on $k$. We start with a Gaussian centered at $\alpha_K=x_K^*$. 
\subsection{Variational Backward Riccati Equation}
We now show that the solution to the problem \eqref{VI-KL} is given by a generalization of the backward Riccati equation (the proof is postponed to Appendix~\ref{AppendixB}).

\begin{proposition} \label{Proposition2}
Consider the dynamic programming recursion \eqref{VI-KL} stemming from problem \eqref{KLcontrolprev0}, with dynamics \eqref{dyna3:eq} and with costs \eqref{cost:eq:q}. Suppose the ``value distribution"  \eqref{phi:def} at previous step is in   the form of a Gaussian $\phi(x_{k+1})=\mathcal{N}(\alpha_{k+1},  {\varepsilon} P_{k+1}^{-1})$ with known parameters $\alpha_{k+1},P_{k+1}$. Then, the optimal joint Gaussian \eqref{Qvalue} for the problem \eqref{VI-KL} satisfies:
\BEQ
\begin{aligned}
&q(x_k) = \mathcal{N}( \alpha_{k},  {\varepsilon} P_{k}^{-1}) \nonumber\\
&q(u_k | x_k) =\mathcal{N}(\beta_{k} + K_k(x_k - \alpha_k) ,  {\varepsilon} S_{k}^{-1}), \nonumber
\end{aligned} 
\EEQ
with $S_k$, $ \beta_{k}$ and $K_k$ given by 
\BEQ
\begin{aligned}
&S_k=R+B^TP_{k+1}B, \quad K_k=-S_k^{-1}B^TP_{k+1}\E_q \Big[\dv{f}{x}(x_k)\Big] \\
&\beta_k=-S_k^{-1}B^TP_{k+1}( \E_q \big[f(x_k) \big]- \alpha_{k+1} ), \label{general:eq0}
\end{aligned} 
\EEQ
and  where $\alpha_k$ and $P_k$ satisfy the generalized (implicit) backward Riccati equation 
 \begin{alignat}{2}
\alpha_k&=x_k^*-Q^{-1} \E_q\Big[\dv{f}{x}(x_k)^\top P_{k+1}  ( f(x_k)+Bu_k- \alpha_{k+1} )   \Big] \nonumber \\
 P_k&=Q-\E_q \Big[ \dv{f}{x}(x_k)\Big]^\top P_{k+1}BS_k^{-1}B^TP_{k+1}\E_q \Big[\dv{f}{x}(x_k)\Big] \nonumber \\
&\hspace*{2cm} +\E_q\Big[\dv{f}{x}(x_k)^\top P_{k+1}\dv{f}{x}(x_k)+H_k\Big ], \label{general:eq}
 \end{alignat}
where $H_k \in \mathcal{M}_d(\mathbb{R})$ is given by the tensor contraction of the Hessian of $f$: $$H_k[\mu,\nu]=\sum_{ij} (P_{k+1})_{ij}( f(x_k)+Bu_k- \alpha_{k+1} )_i\frac{\partial^2 f_j}{\partial x^\mu\partial x^{\nu}}.$$In all the expectancies above, subscript $q$ denotes the joint distribution $q(x_k,u_k).$ \BlackBox
\end{proposition}

The obtained equation resembles the Riccati equation from LQR, but with the presence of expectations. These equations are implicit because the expectations are taken over the sought distribution. This is akin to our prior work in the field of probabilistic inference \cite{rvga}, and various techniques can allow us to get around this issue, as will be discussed in a few paragraphs.

We conclude this subsection with an additional result, proving that when $f$ is odd, the problem simplifies by symmetry (see proof in Appendix~\ref{AppendixB}).

\begin{lemma}\label{lemma1}
   Assume we start with a terminal cost that is centered, in the sense that $\phi(x_{K}):=\exp{(-L_K(x_K)/\varepsilon)}$ is (up to a normalization constant) a centered Gaussian $\mathcal N(0,\varepsilon P_K^{-1})$. Assume additionally $f(-x)=-f(x)$ for all $x$.   Then for all $k< K$ we have $\alpha_k=\beta_k=0 $.  \BlackBox
\end{lemma}

 \subsection{Linear Case}
  In the case of stochastic linear dynamics with quadratic costs, one can wonder what our entropy regularized problem~\eqref{KLcontrolprev0} boils down to and what role the regularization parameter~$\varepsilon$ plays. By applying Proposition \ref{Proposition2}  with  $x^*_k=0$ for $k=1,\dots,K$,  we recover the LQR equations. To be more precise, the control policies--herein defined as Gaussian distributions--have their mean parameters governed by the LQR equations indeed, while their covariance matrices have a magnitude of order $\varepsilon$, which reflects the penalization on the negative entropy that prompts dispersion. 
  
\begin{corollary} \label{Corollary1}
In the linear case $f(x_k)=Ax_k$, the stochastic dynamic \eqref{dyna3:eq} becomes $x_{k+1}=Ax_k+Bu_k+\nu_k$ with $\nu_k\sim \mathcal N(0,C)$. Letting $\alpha_K=0$, the optimal policy and value distributions write:
\BEQ
\begin{aligned}
&&q(x_k) = \mathcal{N}(0,  {\varepsilon} {P_{k}}^{-1}), \quad q(u_k | x_k) =\mathcal{N}(K_kx_k ,  {\varepsilon} {S_{k}}^{-1}), \nonumber
\end{aligned} 
\EEQ
which notably means that the value function writes $V_k(x_k)=\frac{1}{2}x_k^TP_kx_k.$
The parameters are given by
\BEQ
\begin{aligned}
&S_k=R+B^TP_{k+1}B, \quad K_k=-{S_k}^{-1}B^TP_{k+1}A,
\end{aligned} 
\EEQ
and $P_k$ satisfies the classical backward Riccati equation:
 \begin{alignat}{2}
  &P_k=A^\top P_{k+1}A+Q \nonumber\\
 &-A^\top P_{k+1}B(R+B^TP_{k+1}B)^{-1}B^TP_{k+1}A. 
 \end{alignat}
\end{corollary}
\begin{proof}
Since $f(x_k)=Ax_k$ is odd, Lemma \ref{lemma1} shows that $\alpha_k=\beta_k=0$. Moreover $H_k=0$ and replacing $\dv{f}{x}(x_k)$ with $A$ gives the  equations above.
\end{proof} 

\subsection{Discussion}\label{discuss:sec}

In the case of control-affine nonlinear dynamics, and assuming centered distributions to simplify, we see we essentially recover LQR equations where $A$ is replaced with an expectation   of the form
 $$\E_q \Big[\dv{f}{x}(x_k)\Big] =\int \dv{f}{x}(x)\tilde C\exp\Big(-\frac{x^TP_kx}{2\varepsilon }\Big)|P_k|^{-1/2}\frac{1}{\sqrt\varepsilon}dx.$$A change of variables shows this is equal to
 $$
\int\dv{f}{x}(\sqrt \varepsilon y)\tilde C\exp\Big(-\frac{y^TP_ky}{2  }\Big)|P_k|^{-1/2}dy.
 $$  We see the effect of entropy regularization is to perform an average of magnitude $\sqrt \varepsilon$ around the equilibrium (assuming~$0$ is the equilibrium we seek to stabilize), and as $\varepsilon\to 0$ we have $\E_q \Big[\dv{f}{x}(x_k)\Big] \to\dv{f}{x}(0)$, and we recover the LQR equations linearized at equilibrium. 

Note that the equations are implicit. In \eqref{general:eq0},  the definition of $P_k$ is based on an average over $q$, whose variance is $P_k/\varepsilon$, which reminds our previous work on variational inference \cite{rvga}. In practice,    we can cycle as follows for small~$\varepsilon$. We assume $\varepsilon=0$ initially, which gives a first estimate for $P_k$ based on the linearization at equilibrium, as previously explained. Then, we may recompute, letting the obtained~$P_k$ be the variance of $q$. After a few iterations, the scheme converges in practice. 

\begin{remark}
Note that the control gain of our policy \eqref{general:eq0}  is defined by $K_k=-{S_k}^{-1}B^TP_{k+1}\mathbb{E}_q[\dv{f}{x}(x)]$. Taking an average is likely to make the policy more robust to model uncertainty; see, e.g.,  \cite{Ziebart2010}. \end{remark}
\begin{remark}
Another attractive property of our policy is that it allows for the computation of controls when the dynamics $f$ is nondifferentiable. Indeed, we can avoid computing the Jacobian matrix of the dynamics considering instead the Jacobian matrix of the Gaussian: $\mathbb{E}_q[\dv{f}{x}(x)]=-\int \dv{q}{x} (x) f(x) dx $. This equality results from integration by part on the Gaussian $q$, which has a support that vanishes at $\pm \infty$. Nondifferentiable control appears, for example, in collision detection with randomized smoothing   \cite{Montaut2023}. 
\end{remark}
 
\section{Variational Control of a Pendulum} \label{section4}
To illustrate the method and to gain some insight into the obtained optimal solution, we focus on the case study of a pendulum controlled by a torque $u$ and perturbed by a noise $w$. This is a simple example but sufficiently nonlinear to showcase the differences between linearized LQR and entropy-regularized optimal control. The dynamics write
\BEAS
&  \ddot\theta+\lambda \dot \theta  - \omega^2 \sin\theta  = \frac{1}{m \ell^2} u  + \sqrt{\eta}  w,
\EEAS
where $\theta$ is the angle with respect to the pendulum at the unstable equilibrium (upward position), $\omega=\sqrt{g/\ell}$ is the pulsation,  $\lambda=\xi/m$ the damping parameter and $\eta>0$ is the magnitude of the noise. 
In state-space form, the dynamics are discretized in time as follows:
\BEAS
\begin{pmatrix}\theta_{k+1}\\ \dot{\theta}_{k+1}\end{pmatrix}\!\!\!&\!\!\!\!\! = \!\!\!\!\!& \!\begin{pmatrix}\theta_k\\ \dot{\theta}_k\end{pmatrix}+\delta t \begin{pmatrix} \dot{\theta}_k\\ -\lambda \dot \theta_k + \omega^2 \sin \theta_k\end{pmatrix} \\
&&\hspace*{.0cm}+\delta t \begin{pmatrix}0\\ \frac{1}{m \ell^2} \end{pmatrix} u_k + \sqrt{\delta t \eta} \begin{pmatrix}0\\ 1 \end{pmatrix}w,\quad w \sim \mathcal{N}(0,1),\\
&\!\!\!\!\! := \!\!\!\!\!& f(x_k)+Bu_k+\nu_k.
\EEAS
The discrete cost writes
\BEAS
&&   x_K^TQx_K + \sum_{k=0}^{K-1} x_k^TQx_k + u_k^TRu_k.
 \EEAS 
Starting from $\theta_0$ we seek to stabilize the inverted pendulum while penalizing the entropy of the policy. 

We will compare our variational control with the control given by LQR with dynamics linearized around the equilibrium $x^*=0$, that is,  letting  $A=\begin{pmatrix} 1 & \delta t  \\  \delta t  g/\ell & 1 - \delta t  \lambda \end{pmatrix}$.

\subsection{Computation of the Solution}
Since the dynamics of the inverted pendulum satisfy the oddness condition of Lemma \ref{lemma1}, we have $\alpha_k=0$ and $\beta_k=0$, and the optimal policy is given by $q(u_k|x_k)=\mathcal{N}\big(K_kx_k ,\varepsilon S_k^{-1}\big)$ with $K_k$ and $S_k$ defined in Proposition~\ref{Proposition2}. To compute this optimal policy, there are two hurdles: the variational Riccati equation \eqref{general:eq} is implicit, and there are expectations to compute.  As already mentioned in Section~\ref{discuss:sec},  to cope with the fact the equation is implicit, we may open the loop and iterate on the equation in an inner loop. As concerns the expectations under Gaussians, they are approximated using quadrature rules:  
$$\int \mathcal{N}(\mu,P) g(x) dx \approx \sum_{i=1}^{M} w_i g(x_i),$$
where we can choose $M=2d$ cubature points \cite{Arasaratnam09} defined by $w_i=\frac{1}{2d}$ and $x_i=\mu+\sqrt{d} L e_i$ where $e_i$ are basis vectors in dimension $d$, and $L$ the square root matrix of the covariance such that the points are equally spread at the edge of the Gaussian ellipsoid.

\subsection{Numerical Results}
We take the following parameters: $g=9.8, m=1, \ell=1$ and $\xi=1$.  We start at $\theta_0=\frac{\pi}{6}$ and $\dot \theta=0$, and we want to put the pendulum at $(\theta, \dot \theta)=(0,0)$ which corresponds by convention to the unstable equilibrium (upward position) such that the stable equilibrium (downward position) is at $\theta= {\pi} $.  We consider a backward pass with $1000$ iterations with stepsize  $\delta t=0.01$ such that the temporal horizon is $T=10s$. We simulate the Brownian motion with a Gaussian increment of covariance $\delta t \eta$ where $\eta=0.02$ $rd/s^2$ in the first experiment and $\eta=0.2$ $rd/s^2$ in the second one. The forward trajectory is simulated with a semi-implicit Euler-Maruyama scheme to better conserve the system’s energy. For the ``variational control,'' the implicit Riccati backward equation is iterated $10$ times in an inner loop; however,  we found out that one iteration could be used in practice without much affecting the results.

\paragraph{Average control}

We first apply the average value of the policy distribution by letting $u_k:=K_kx_k$.  Figure~\ref{Fig-XP1} illustrates the behavior in function of  $\sqrt{\varepsilon}$, and compares it to LQR control based on the system linearized at the equilibrium. We see clearly that for the smaller value of $\sqrt{\varepsilon}$, both controllers behave similarly, but when $\sqrt{\varepsilon}$ increases, the gains with the variational control are below the LQR gains, leading to softer controls based on averaging a trigonometric function around its maximum (softer controls may preserve actuators). To   underline the effect of $\varepsilon$, we have considered a small cost: $R=\delta t 0.01\mathbb{I}_m$, $Q=\delta t  0.01\mathbb{I}_d$ and $P_K=\delta t  \mathbb{I}_d$.

\begin{figure}[thpb]
\centering
\includegraphics[scale=0.41]{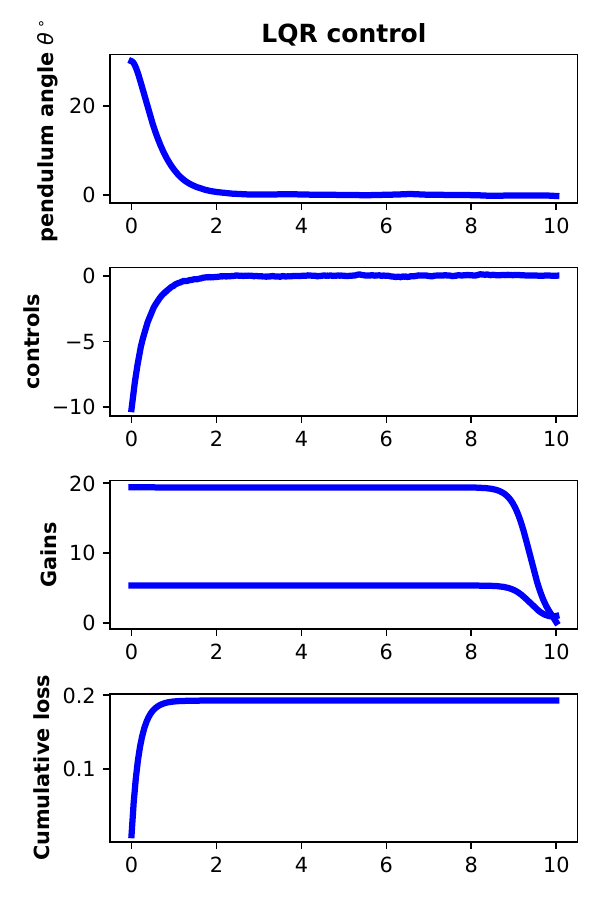}
\includegraphics[scale=0.41]{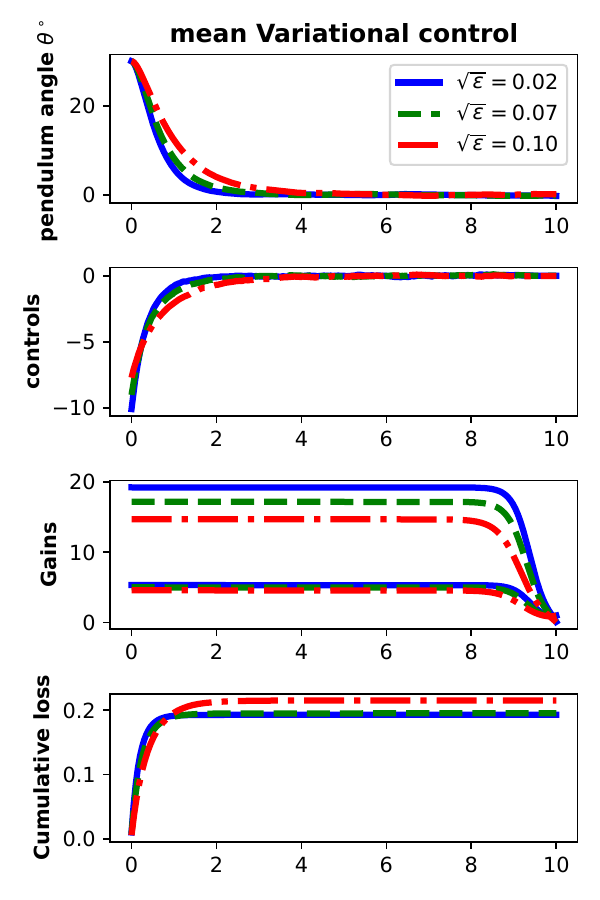}

\vspace*{-.3cm}

\caption{Linearized LQR at equilibrium versus our ``variational control''  for the inverted pendulum regulation where we apply for the KL the mean policy $\mathbb{E}[u_k |x_k]$ for different values of entropic regulation $\sqrt{\varepsilon}=0.02, 0.07, 0.10$. From the top to the bottom row, we show the angle $\theta$ converted in degrees, the control, the two gains for angle and angular velocity, and finally, we compare both LQR and variational control with the same LQR quadratic loss.  }
\label{Fig-XP1}
\end{figure}

\paragraph{Random control}
We now sample the control from the actual policy distribution $q(u_k | x_k) =\mathcal{N}(K_kx_k,  \varepsilon (R+B^TP_{k+1}B)^{-1})$. We consider a large terminal cost   $ P_K=\delta t  1000 \mathbb{I}_d$  but low stage costs   $R=\delta t 0.01\mathbb{I}_m$, $Q=\delta t  0.01\mathbb{I}_d$. In this way, we elicit high entropy along the path (hence exploration of the state space) while enforcing the final equilibrium state. Results are displayed in figure \ref{Fig-XP2}, where we see the empirical distribution of the state when applying random controls. The distribution $p(x_k)$ of the state spreads during the transient phase but shrinks to the equilibrium indeed at the final time $T$. 

\begin{figure}[thpb]
\centering
\includegraphics[scale=0.41]{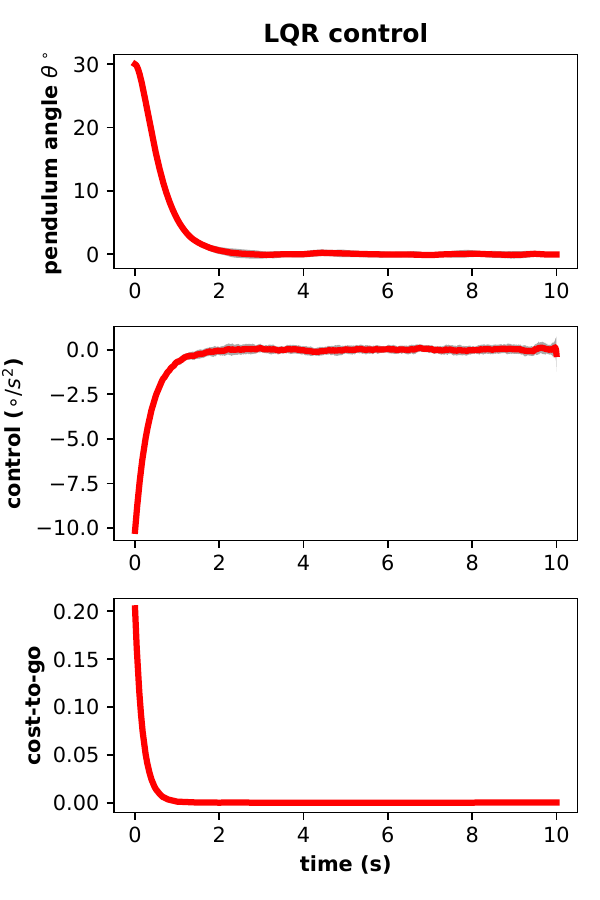}
\includegraphics[scale=0.41]{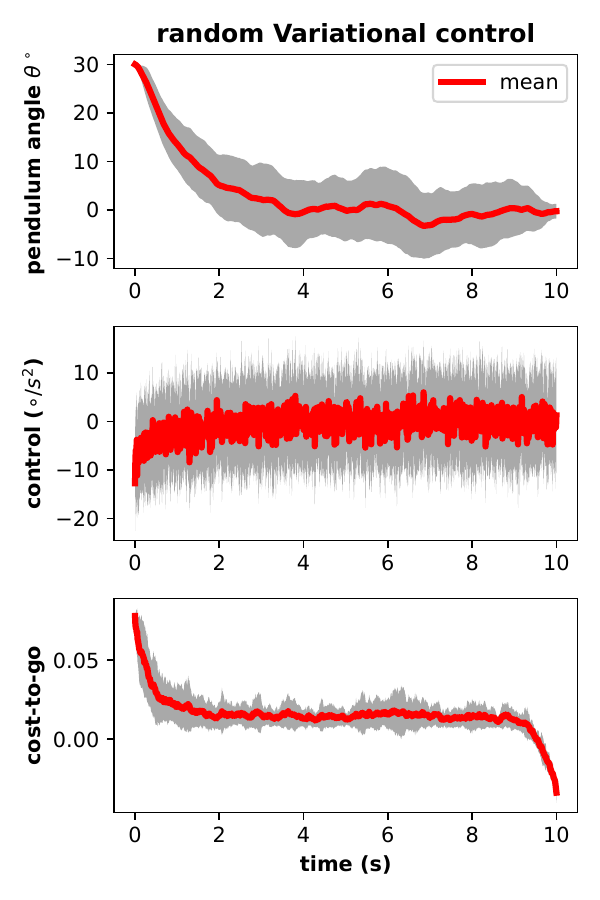}

\vspace*{-.4cm}

\caption{Linearized LQR versus ``Variational control'' for the inverted pendulum regulation where we sample randomly from the policy $q(u_k|x_k)$. The entropic regulation parameter is fixed to $\sqrt{\varepsilon}=0.10$. We execute $30$ Monte Carlo runs, and for each output,  we draw the empirical mean in red and the empirical standard deviation in grey. From the top to the bottom row, we show the angle $\theta$ converted in degrees, the control, and the cost-to-go at the current state. The cost-to-go at $x_k$ is defined by $\frac{1}{2}x_k^TP_kx_k$ for LQR and by $-\varepsilon \log q(x_k)=-\varepsilon \log \mathcal{N}(x_k|0,\varepsilon P_k^{-1})$ for variational control. }
\label{Fig-XP2}
\end{figure}

\textit{The sources of the code are available on Github on the following repository:\\ \url{https://github.com/marc-h-lambert/KL-control}.}

\addtolength{\textheight}{-3cm}   

\section{CONCLUSION}
We have proposed a new setting for stochastic optimal control based on the entropic regularization of both the entropy of the dynamics and the entropy of the policy. This problem was reformulated as a KL divergence between two processes: the first defining the controlled stochastic trajectory, the second defining a reward process. Following a variational dynamic programming principle, we have shown we can compute the exact optimal policy and cost-to-go. However, in practice, it is impractical, and both the control policy and the cost-to-go need to be approximated. We showed that they can be approximated jointly with a Gaussian distribution. In the case of nonlinear dynamics with affine control inputs and quadratic costs, this approximation can be computed in closed form, leading to tractable formulas that generalize the backward Riccati equation from LQR control. 

To illustrate the results, we have performed simulations using our new policy on a second-order system. Using the average policy results in softer control with smaller gains than LQR, whereas the random policy causes dispersion in the state space during the transient phase.

The proposed method paves the way for future work: the control affine model can be made richer by considering a state-dependent control matrix $B(x)$ and a state-dependent covariance of Brownian motion $C(x)$. Moreover, we could use richer approximating families, such as mixtures of Gaussians, to more closely capture the value function. 


\section{ACKNOWLEDGMENTS}
This work was funded by the French Defence Procurement Agency (DGA) and by the French government under the management of Agence Nationale de la Recherche as part of the “Investissements d’avenir” program, reference ANR-19-P3IA-0001 (PRAIRIE 3IA Institute).


\bibliographystyle{IEEEtran}
\bibliography{IEEEabrv,KLcontrol}

\begin{thebibliography}{10}
\providecommand{\url}[1]{#1}
\csname url@rmstyle\endcsname
\providecommand{\newblock}{\relax}
\providecommand{\bibinfo}[2]{#2}
\providecommand\BIBentrySTDinterwordspacing{\spaceskip=0pt\relax}
\providecommand\BIBentryALTinterwordstretchfactor{4}
\providecommand\BIBentryALTinterwordspacing{\spaceskip=\fontdimen2\font plus
\BIBentryALTinterwordstretchfactor\fontdimen3\font minus \fontdimen4\font\relax}
\providecommand\BIBforeignlanguage[2]{{%
\expandafter\ifx\csname l@#1\endcsname\relax
\typeout{** WARNING: IEEEtran.bst: No hyphenation pattern has been}%
\typeout{** loaded for the language `#1'. Using the pattern for}%
\typeout{** the default language instead.}%
\else
\language=\csname l@#1\endcsname
\fi
#2}}

\bibitem{Ziebart2008}
B.~D. Ziebart, A.~Maas, J.~A. Bagnell, and A.~K. Dey, ``Maximum entropy inverse reinforcement learning,'' \emph{AAAI Conference on Artificial Intelligence}, pp. 1433--1438, 2008.

\bibitem{Haarnoja2018}
T.~Haarnoja, A.~Zhou, P.~Abbeel, and S.~Levine, ``Soft actor-critic: Off-policy maximum entropy deep reinforcement learning with a stochastic actor,'' \emph{35th International Conference on Machine Learning}, 2018.

\bibitem{Ziebart2010}
B.~D. Ziebart, J.~A. Bagnell, and A.~K. Dey, ``Modeling interaction via the principle of maximum causal entropy,'' \emph{International Conference on Machine Learning}, pp. 1255--1262, 2010.

\bibitem{Oswin2022}
O.~So, Z.~Wang, and E.~A. Theodorou, ``Maximum entropy differential dynamic programming,'' \emph{International Conference on Robotics and Automation}, pp. 3422--3428, 2022.

\bibitem{Wainwright08}
M.~J. Wainwright and M.~I. Jordan, ``Graphical models, exponential families, and variational inference,'' \emph{Foundations and Trends in Machine Learning}, vol.~1, no. 1–2, pp. 1--305, 2008.

\bibitem{Todorov2009}
E.~Todorov, ``Efficient computation of optimal actions,'' \emph{Proceedings of the National Academy of Sciences}, pp. 11\,478--11\,483, 2009.

\bibitem{Kappen2012}
H.~J. Kappen, V.~G\'{o}mez, and M.~Opper, ``Optimal control as a graphical model inference problem,'' \emph{Machine Language}, pp. 159--182, 2012.

\bibitem{Theodorou2012}
E.~A. Theodorou and E.~Todorov, ``Relative entropy and free energy dualities: Connections to path integral and {KL} control,'' \emph{Conference on Decision and Control}, pp. 1466--1473, 2012.

\bibitem{Fleming06}
W.~H. Fleming, ``Risk sensitive stochastic control and differential games,'' \emph{Communications In Information And Systems}, 2006.

\bibitem{Chen2021}
Y.~Chen, T.~T. Georgiou, and M.~Pavon, ``Stochastic control liaisons: {R}ichard {S}inkhorn meets {G}aspard {M}onge on a {S}chr{\"o}dinger bridge,'' \emph{SIAM Review}, pp. 249--313, 2021.

\bibitem{Rawlik07}
K.~Rawlik, M.~Toussaint, and S.~Vijayakumar, ``On stochastic optimal control and reinforcement learning by approximate inference,'' \emph{International Joint Conference on Artificial Intelligence}, 2012.

\bibitem{Giovanni22}
D.~Gagliardi and G.~Russo, ``On a probabilistic approach to synthesize control policies from example datasets,'' \emph{Automatica}, vol. 137, 01 2022.

\bibitem{Todorov08}
E.~Todorov, ``General duality between optimal control and estimation,'' \emph{Conference on Decision and Control}, 2008.

\bibitem{Toussaint2009}
M.~Toussaint, ``Robot trajectory optimization using approximate inference,'' \emph{International Conference on Machine Learning}, pp. 1049--1056, 2009.

\bibitem{rvga}
M.~Lambert, S.~Bonnabel, and F.~Bach, ``The recursive variational gaussian approximation ({R-VGA}),'' \emph{Statistics and Computing}, vol.~32, 2022.

\bibitem{spong2005underactuated}
M.~W. Spong, ``Underactuated mechanical systems,'' in \emph{Control Problems in Robotics and Automation}.\hskip 1em plus 0.5em minus 0.4em\relax Springer, 2005, pp. 135--150.

\bibitem{Montaut2023}
L.~Montaut, Q.~L. Lidec, A.~Bambade, V.~Petr{\'{\i}}k, J.~Sivic, and J.~Carpentier, ``Differentiable collision detection: a randomized smoothing approach,'' \emph{International Conference on Robotics and Automation}, pp. 3240--3246, 2023.

\bibitem{Arasaratnam09}
I.~Arasaratnam and S.~Haykin, ``Cubature {K}alman filters,'' \emph{IEEE Trans. Automat. Control}, vol.~54, no.~6, pp. 1254--1269, 2009.

\end{thebibliography}
 
\section{Appendix A: Proof of Proposition \ref{Proposition1}} \label{AppendixA}
To prove the Proposition \ref{Proposition1} we use the 
following Lemma:
\begin{lemma}\label{lem22}
Let’s consider the following problem: 
\begin{alignat*}{2}
\underset{p(z) \in \mathcal{P}(\mathbb{R}^d)}{\min}  &{\rm KL}(p(z)p(x|z) \| h(x,z))\\
:=\underset{p(z) \in \mathcal{P}(\mathbb{R}^d)}{\min} &\int \int p(z)p(x|z)  \log \frac{p(z)p(x|z) }{h(x,z)}dxdz,
\end{alignat*}
where $p(z)$ and $p(x|z)$ are probability distributions and $h$ is a density function which may be unnormalized.  $\mathcal{P}(\mathbb{R}^d)$ is the space of probability distribution smoothed enough to admit a density function. Then, the minimum is attained at 
\begin{alignat*}{2}
&p^*(z) = \frac{1}{Z} \exp \Big(-\int p(x|z)  \log \frac{p(x|z) }{h(x,z)}dx \Big),
\end{alignat*}
where $Z$ is the normalization constant of $p^*(z)$. Moreover, the minimum is $-\log Z$. \BlackBox
\end{lemma}
\begin{proof}
\begin{alignat*}{2}
& \int \int p(z) p(x|z) \log \frac{p(x|z)p(z) }{h(z,x)} dx dz\\[-.1cm]
&=\int p(z) \int p(x|z) \log p(z)dxdz \\[-.2cm]
&+ \int p(z) \int p(x|z) \log \frac{p(x|z)}{h(z,x)} dx dz\\[-.1cm]
&=\int p(z) \log p(z) dz - \int p(z) \log f(z) dz\\[-.1cm]
&\quad  \text{where} \quad f(z)=\exp \Big(-\int p(x|z) \log \frac{p(x|z)}{h(z,x)} dx \Big)\\[-.1cm]
&=KL(p(z)||f(z)) \quad \text{which is minimal for $p(z) \propto f(z)$ } \\[-.1cm]
&=-\log Z \quad \text{where $Z= \int f(z) dz$}. \\[-1.25cm]
\end{alignat*}
\end{proof}
Applying Lemma \ref{lem22} to  $z=u_k|x_k$, $x=x_{k+1}$ and $h(z,x)=r(x_k,u_k)\phi^*_{k+1}(x_{k+1})$, we obtain the desired result.

\section{Appendix B: Proof of Proposition \ref{Proposition2} and Lemma \ref{lemma1}} 
\label{AppendixB}
To show proposition \ref{Proposition2}, we first reformulate the problem \eqref{VI-KL} for our particular control-affine setting.
\paragraph{Reformulation of the problem}
 \begin{alignat*}{2}
&{\rm KL}(q(x_k,u_k)p(x_{k+1}|x_{k},u_k) \| \exp(-\ell(x_k,u_k)/\varepsilon)q(x_{k+1}))\\
&=-H(q(x_k,u_k)) - H(p(x_{k+1}|x_{k},u_k)) \\
&+ \int q(x_k,u_k) \frac{1}{\varepsilon}  \ell(x_k,u_k) dx_k du_k\\
&-  \int q(x_k,u_k) \int p(x_{k+1}|x_k,u_k) \log q(x_{k+1}) dx_{k+1} dx_k du_k,
\end{alignat*}
where $H$ is the  entropy operator which writes  $H(q(x_k,u_k))=\frac{1}{2}\log |\Sigma_k| + c$ and $H(p(x_{k+1}|x_{k},u_k))=c^\prime$ where $c$ and $c^\prime$ are constants independent of the variational parameters. 
The last integral on $p(x_{k+1}|x_k,u_k)$ simplifies as follows, denoting $p(x_{k+1}|x_k,u_k)$ by $p$:
 \begin{alignat*}{2}
 &\int p(x_{k+1}|x_k,u_k) \log q(x_{k+1}) dx_{k+1}:=\E_p[\log q(x_{k+1})]\\
 & =\E_p[\frac{1}{2 \varepsilon} (x_{k+1}-\alpha_{k+1})^\top P_{k+1}(x_{k+1}-\alpha_{k+1})]\\
& =\frac{1}{2 \varepsilon} (f(x_k)+Bu_k-\alpha_{k+1})^\top P_{k+1}(f(x_k)+Bu_k-\alpha_{k+1})\\
&+\E_p[\nu_k^\top P_{k+1}\nu_k],
\end{alignat*}
where $\E_p[\nu_k^\top P_{k+1}\nu_k]=\tr CP_{k+1}$ interestingly does not depend on variational parameters.
Finally,   \eqref{VI-KL}  reduces to:
 \begin{alignat}{2}
\underset{\mu_k,\Sigma_k}{\min} \quad \E_q[g(x_k,u_k)] -\frac{1}{2}\log |\Sigma_k|, \label{VI-KL2}
\end{alignat}
where $ \E_q$  denotes the expectation under $q(x_k,u_k)={\mathcal{N}(\mu_k,\Sigma_k)}$ and where $g$ is defined  as follows:
 \begin{alignat*}{2}
&g(x_k,u_k)  =    \frac{1}{2\varepsilon} ((x_k-x_k^*)^TQ(x_k-x_k^*)+u_k^TRu_k)\\
&+   \frac{1}{2\varepsilon} (  f(x_k)+Bu_k- \alpha_{k+1} )^\top P_{k+1} 
( f(x_k)+Bu_k- \alpha_{k+1} ).
\end{alignat*}

\paragraph{Closed from solution}
To solve \eqref{VI-KL2}, we use the property of integration under Gaussian distribution described in the following result known as Stein's Lemma: 
\begin{lemma} For a function $f: \rb^d \to \rb$,
\BEAS
&&\nabla_{\mu} \int \mathcal{N}(x|\mu,\Sigma) f(x) dx  =  \int \mathcal{N}(x|\mu,\Sigma)\nabla  f(x) dx  \\
&&\nabla_{\Sigma} \int \mathcal{N}(x|\mu,\Sigma) f(x) dx   =  \frac{1}{2} \int \mathcal{N}(x|\mu,\Sigma)\nabla^2 f(x)  dx.
\EEAS 
\end{lemma}
\begin{proof}
The proof comes from integration by part and using the symmetric properties of Gaussians $\nabla_{\mu} \mathcal{N}(x|\mu,\Sigma) = -\nabla_{x} \mathcal{N}(x|\mu,\Sigma) $ and $\nabla_{\Sigma} \mathcal{N}(x|\mu,\Sigma) = \frac{1}{2} \nabla^2_{x} \mathcal{N}(x|\mu,\Sigma)$.
\end{proof}
Using this lemma and the relation $\nabla_{\Sigma}  \log |\Sigma| = \Sigma^{-1}$, the derivative  with respect to $\Sigma_k$ of the quantity \eqref{VI-KL2} writes:
$$
 \frac{1}{2} \E_q
\bigg[
\bimat{ \nabla_{xx} g_k(x_k,u_k)}{ \nabla_{xu} g_k(x_k,u_k)}{ \nabla_{ux} g_k(x_k,u_k)}{ \nabla_{uu} g_k(x_k,u_k)}
\bigg]-\frac{1}{2}\Sigma_k^{-1}.
$$
Writing $\nabla_{\Sigma}(\cdot)=0$ yields for the problem at hand
\BEAS
\Sigma_k^{-1}=\frac{1}{\varepsilon}  \bigg[
\bimat{ \varepsilon\E_q \Big[ \nabla_{xx} g(x_k,u_k) \Big]}{\E_q\Big[ \dv{f}{x}(x_k)^T\Big]P_{k+1}B}{ B^TP_{k+1}\E_q \Big[\dv{f}{x}(x_k) \Big] }{ R+B^TP_{k+1}B}
\bigg].
\EEAS
Recalling our model for the joint covariance as a $2 \times 2$ block matrix $\Sigma_k$ \eqref{Qvalue}, we can compare the above matrix with the inverse $\Sigma_k^{-1}$ given by :  
\begin{alignat}{2}
 &\Sigma_k^{-1} = \varepsilon^{-1} \bimat{P_k + K_k^\top S_k K_k }{-K_k^\top S_k}{ -S_k K_k}{S_k}. \label{invQvalue}
\end{alignat}
By identification, this readily yields
 \begin{alignat}{2}
S_k&=R+B^TP_{k+1}B \nonumber\\
-S_kK_k&=B^TP_{k+1}\E_q \Big[\dv{f}{x}(x_k)\Big] \nonumber\\
P_k + K_k^\top S_k K_k&=Q+\E_q \Big[ \dv{f}{x}(x_k)^TP_{k+1}\dv{f}{x}(x_k)+H_k \Big], \label{Hessian-form}
 \end{alignat}
 where the last equation comes from a computation of the upper left term $\E_q \Big[ \nabla_{xx} g_k(x_k,u_k) \Big]$. 
We then deduce the expression for $K_k$ and $P_k$. 

From Stein's lemma, the derivative  w.r.t. $\mu_k$ of  \eqref{VI-KL2} is   $\Big(\E_q \Big[\dv{g}{x}(x_k)\Big],\E_q \Big[ \dv{g}{u}(u_k) \Big]\Big)$. Setting it to zero gives :
\small
 \begin{alignat*}{2}
0&=Q (\alpha_k-x_k^*)+\E_q\Big[\dv{f}{x}(x_k)^TP_{k+1}  ( f(x_k)+Bu_k- \alpha_{k+1} )   \Big ]\\
0&=\frac{1}{\varepsilon}(R\beta_k+B^TP_{k+1}B\beta_k+ B^TP_{k+1}( \E_q \big[ f(x_k) \big]- \alpha_{k+1} ) ) 
 \end{alignat*}
 
 \normalsize
 from which we deduce the expression for $\alpha_k$ and $\beta_k$. 
\paragraph{Proof of Lemma \ref{lemma1} \label{proofLemma1}}
We now show the general equations \eqref{general:eq0}-\eqref{general:eq} may be simplified under oddness conditions. 
Assume $\alpha_{k+1}=0$. We let $\alpha_k=0$ and $\beta_k=0$, and we want to show the equations on $\alpha_k,\beta_k$ are satisfied. By doing so, we are dealing with centered expectancies. We have $\E [f(x_k)]=0$, proving $\beta_k=0$ is consistent with $\alpha_k=0$. 
Besides, we have  $\dv{f}{x}(-x)=\dv{f}{x}(x)$, which entails   $\alpha_k=0\Rightarrow \E [\dv{f}{x}(x_k)^TP_{k+1}  f(x_k)]=0.$ Finally, we write using the law of total expectation
\BEAS
&&\E_{q(x_k,u_k)} \Big[\dv{f}{x}(x_k)^TP_{k+1}  Bu_k \Big]\\
&&=\E_{q(x_k)} \Big[ \dv{f}{x}(x_k)^TP_{k+1}  B\E [u_k\mid x_k] \Big]\\
&&=\E_{q(x_k)} \Big[\dv{f}{x}(x_k)^TP_{k+1}  BK_k x_k \Big],
\EEAS and we use the fact we integrate an odd function w.r.t. a centered Gaussian.
\end{document}